\documentclass[11pt]{amsart}

\usepackage{amsfonts, amssymb, amsmath, graphicx}
\usepackage[all, knot]{xy}

\input xy

\makeatletter
\@namedef{subjclassname@2020}{%
  \textup{2020} Mathematics Subject Classification}
\makeatother

\newtheorem{theorem}{Theorem}
\theoremstyle{definition}

\newtheorem{lemma}{Lemma}
\newtheorem{proposition}[theorem]{Proposition}

\newtheorem{remark}[theorem]{Remark}

\begin{document}

\title{On an invariant for colored classical and singular links}

\author{Audrey Baumheckel}
\address{Department of Mathematics, California State University, Fresno \linebreak 5245 North Backer Avenue, M/S PB 108, CA 93740, USA}
\email{abaumheckel@mail.fresnostate.edu}

\author{Carmen Caprau}
\address{Department of Mathematics, California State University, Fresno\linebreak 5245 North Backer Avenue, M/S PB 108, CA 93740, USA}
\email{ccaprau@csufresno.edu}

\author{Conor Righetti}
\address{Department of Mathematics, California State University, Fresno \linebreak 5245 North Backer Avenue, M/S PB 108, CA 93740, USA}
\email{conorrighetti@mail.fresnostate.edu}

\subjclass[2020]{57K12, 57K14}
\keywords{links, singular links, polynomial invariants for links}
\thanks{CC was partially supported by Simons Foundation grant $\#355640$ and NSF grant DMS-$2204386$}

\begin{abstract}
A colored link, as defined by Francesca Aicardi, is an oriented classical link together with a ‘coloration’, which is a function defined on the set of link components and whose image is a finite set of ‘colors’.  An oriented classical link can be regarded as a colored link with its components colored with a sole color. Aicardi constructed an invariant $F(L)$ of colored links $L$ defined via skein relations.
When the components of a colored link are colored with the same color or when the colored link is a knot, $F(L)$ is a specialization of the HOMFLY-PT polynomial. Aicardi also showed that $F(L)$ is a stronger invariant than the HOMFLY-PT polynomial when evaluated on colored links whose components have different colors.
In this paper, we provide a state-sum model for the invariant $F(L)$ of colored links using a graphical calculus for oriented, colored, 4-valent planar graphs. We also extend $F(L)$ to an invariant of oriented colored singular links.
\end{abstract}

\maketitle


\section{Introduction} \label{intro}

Francesca Aicardi~\cite{A2015, A2016} introduced a special type of oriented  links that she called \textit{colored links}. If a function $\gamma$ is defined on the set $C$ of components of an oriented link $L$ and assigns elements of a finite set $N$ of `colors', then the link $L$ is called \textit{colored} and the function $\gamma$ is called a \textit{coloration}. If $|\gamma(C)| = 1$, then the colored link is a classical link; in this case, all components of $L$ are colored with the same color. A coloration $\gamma$ partitions the set $C$ of components of the oriented link $L$, hence $\gamma$ introduces an equivalence relation on $C$. Two colorations $\gamma$ and $\gamma'$ of $L$ are called \textit{equivalent} if there exists a bijection $\gamma (C) \to \gamma'(C)$. 

An invariant of colored links is an invariant of links that takes the same value on ambient isotopic links with equivalent colorations. 

Colored links are closely related to tied links. A \textit{tied link}~\cite{AJ2016} is an oriented classical link whose components may be connected by ties. A \textit{tie} is an arc that connects two points on the link, where the two points can belong to different components or the same component of the link. A tie connecting points that belong to the same component of a link can be removed, and such removal does not change the type of tied link. Ties are merely notational devices depicted as springs and are not embedded arcs. The set of ties on a tied link partitions the set $C$ of components of a link into equivalence classes. Hence tied links are links whose components are partitioned into classes. Equivalently, one can use colors to indicate the components of the link that are in the same equivalence class. In this sense, tied links are equivalent to colored links. In this paper, we chose to work with colored links for the sake of having less crowded diagrams.

Another type of knotted objects that we work with in this paper are oriented singular links. A \textit{singular link} is an immersion in $\mathbb{R}^3$ of a disjoint union of circles, such that the immersion contains finitely many singularities which are transverse double points.  A diagram of a singular link is a projection of the singular link in a plane; such a diagram may contain classical crossings, as in classical link diagrams, and singular crossings. A singular crossing in a diagram is depicted as a 4-valent vertex.

It is well known that two singular link diagrams represent \textit{equivalent} (or \textit{ambient isotopic}) singular links if and only if there exists a finite sequence of the extended Reidemeister moves mapping one diagram onto the other (see for example~\cite{Ka1989, Ka2001}). A proof of this statement can be found in~\cite{Caprau2016}. The \textit{extended Reidemeister moves}, depicted in Fig.~\ref{ExtendedReidMoves}, are local moves on singular link diagrams and include the classical Reidemeister moves $R1, R2$ and $R3$, together with the additional $R4$ and $R5$ moves involving singular crossings.
\begin{figure}[ht]  
\[\raisebox{-10pt}{\includegraphics[height=0.38in]{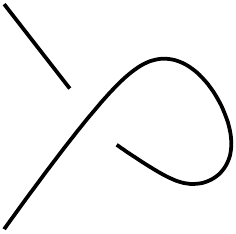}}\, \stackrel{R1}{\longleftrightarrow} \, \raisebox{-10pt}{\includegraphics[height=0.38in]{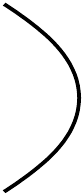}} \,\stackrel{R1}{\longleftrightarrow} \,  \reflectbox{\raisebox{17pt}{\includegraphics[height=0.38in, angle = 180]{poskink}}} \,\qquad\,
 \raisebox{-12pt}{\includegraphics[height=0.38in]{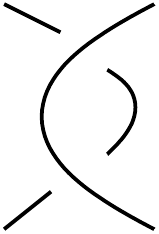}} \, \stackrel{R2}{\longleftrightarrow} \, \raisebox{-12pt}{\includegraphics[height=0.38in]{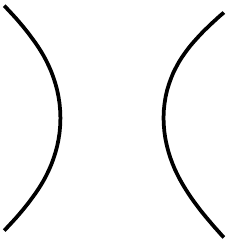}} \,\qquad \,
\raisebox{-12pt}{\includegraphics[height=0.38in]{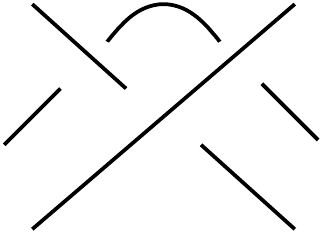}} \, \stackrel{R3}{\longleftrightarrow} \,   \raisebox{-12pt}{\includegraphics[height=0.38in]{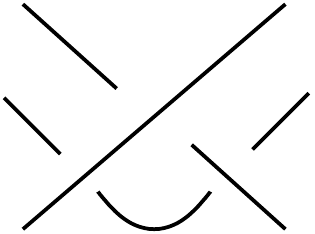}}\]

\[  \raisebox{-12pt}{\includegraphics[height=0.4in]{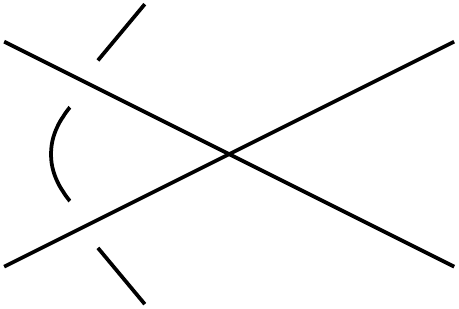}}\,\, \stackrel{R4}{\longleftrightarrow} \,\, \raisebox{-12pt}{\includegraphics[height=0.4in]{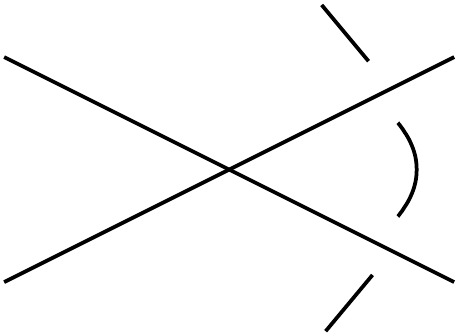}} \,\, \qquad \,\,  \raisebox{-12pt}{\includegraphics[height=0.4in]{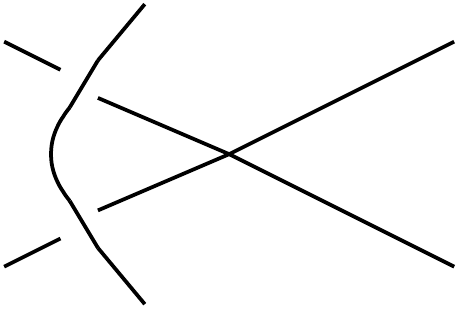}} \,\, \stackrel{R4}{\longleftrightarrow}\,\,  \raisebox{-12pt}{\includegraphics[height=0.4in]{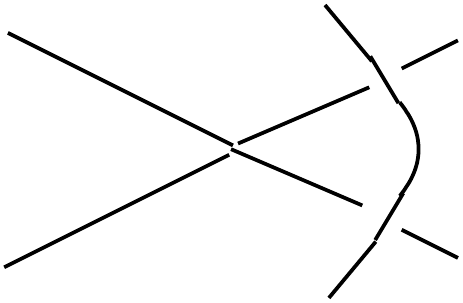}} \]\\
\[ \raisebox{-10pt}{\includegraphics[height=0.35in]{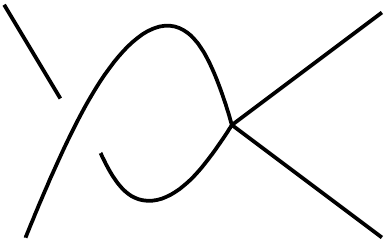}}\,\,\stackrel{R5}{ \longleftrightarrow} \,\,\reflectbox{\raisebox{-10pt}{\includegraphics[height=0.35in]{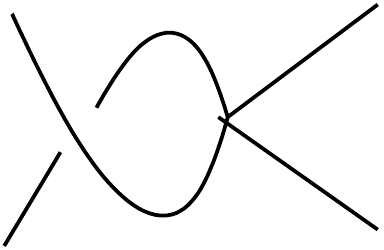}}}\]
\caption{Extended Reidemeister moves} \label{ExtendedReidMoves}
\end{figure}

We remark that singular links can also be regarded as rigid-vertex embeddings in $\mathbb{R}^3$  of 4-valent graphs. For details, we refer the reader to Kauffman's work in~\cite{Ka1989}.

  The concept of coloring can be extended to oriented singular links as well. We define the notion of \textit{colored singular link} as an oriented singular link together with a coloration defined on the components of the singular link. An invariant for colored singular links must take the same value on ambient isotopic singular links with equivalent colorations.  The diagram on the left in Fig.~\ref{ExColoredLinks} is an example of a colored (classical) link with three components, while the diagram on the right is an example of a colored singular link with two components.
\begin{figure}[ht] 
 \includegraphics[height=1in]{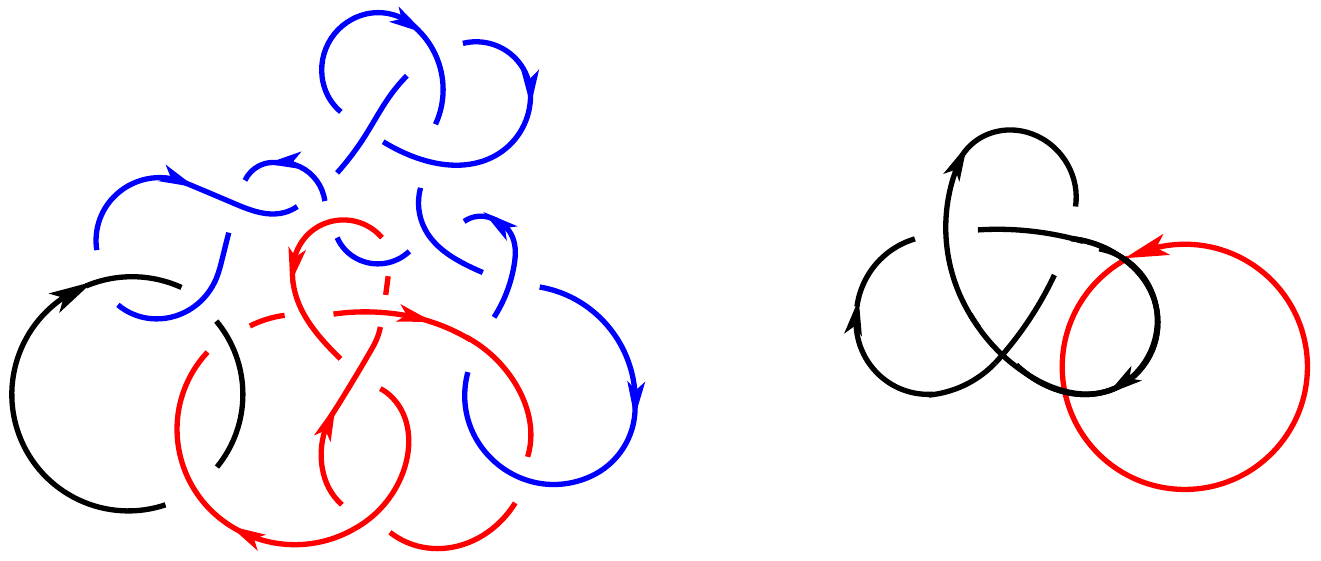}
 \caption{Examples of colored classical and singular links} \label{ExColoredLinks}
\end{figure}

One way to consider links as colored links or singular links as colored singular links is by coloring all of the link components with the same color. Another way is by coloring all of the link components with different colors.

Aicardi~\cite{A2016} defined an invariant $F(L)$ for colored links $L$. We note that a similar invariant for tied links was defined by Aicardi and Juyumaya in~\cite{AJ2016}. When the link components are colored with a sole color or when the link is a knot, the invariant $F$ is a specialization of the HOMFLY-PT polynomial~\cite{HOMFLY, PT} for oriented classical links. The reason for studying the polynomial $F$ is that it is a stronger invariant than the HOMFLY-PT polynomial when $F$ is evaluated on colored links whose components have different colors; for details, we refer to Aicardi's work in~\cite{A2015}. 

Given a colored link $L$, the invariant $F(L)$ is valued in $\mathbb{Q}(x, t, w)$ and is uniquely determined by the following three conditions:

 (I) $F \left( \raisebox{-.04 in}{\includegraphics[scale=0.2]{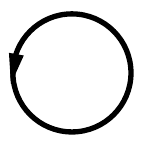}} \right) = 1$.

 (II) Let $D \cup \raisebox{-0.04 in}{\includegraphics[scale=0.2]{unknot.pdf}}$ be the disjoint union of an unknotted circle with a colored link diagram $D$, where the colors of $D$'s components are all distinct from the color of the circle.  Then 
\[ F \left( D \cup \raisebox{-0.04 in}{\includegraphics[scale=0.2]{unknot.pdf}}\right) = \frac{1}{wx} F \left( D \right).\]
(III) The following skein relation holds,
\[ \frac{1}{w} F\left( \reflectbox{\raisebox{-.1 in}{\includegraphics[scale=0.3]{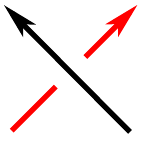}}} \right) - w F\left( \raisebox{-.1 in}{\includegraphics[scale=0.3]{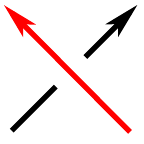}} \right) = \left( 1- \frac{1}{t} \right) F\left( \raisebox{-.1 in}{\includegraphics[scale=0.3]{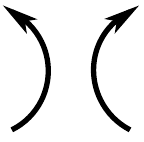}} \right) + \left( \frac{1}{w} - \frac{1}{tw} \right) F\left(\reflectbox{\raisebox{-.1 in}{\includegraphics[scale=0.3]{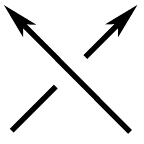}}} \right),\]
where when forgetting the colors, the four diagrams shown in the skein relation above are parts of link diagrams that are identical almost everywhere, except in a small neighborhood where they differ as shown. The diagrams \reflectbox{\raisebox{-.1 in}{\includegraphics[scale=0.3]{poscrosscolor_reflect.pdf}}} and \raisebox{-.1 in}{\includegraphics[scale=0.3]{negcrosscolor.pdf}} are parts of colored links, where black and red colors indicate any colors, not necessarily distinct. Moreover, the components of the colored links that contain the parts \raisebox{-.1 in}{\includegraphics[scale=0.3]{uncross.pdf}}, \reflectbox{\raisebox{-.1 in}{\includegraphics[scale=0.3]{negcross.pdf}}}, and \raisebox{-.1 in}{\includegraphics[scale=0.3]{negcross.pdf}} are colored with a sole color.

The skein relation (III) holds for any two colors of the strands in the two diagrams on the left of the relation. If the two colors are the same, then relation (III) is reduced to the following skein relation:
\[ (\text{IV}) \hspace{1cm} \frac{1}{tw} F\left( \reflectbox{\raisebox{-.1 in}{\includegraphics[scale=0.3]{negcross.pdf}}} \right) - w F\left( \raisebox{-.1 in}{\includegraphics[scale=0.3]{negcross.pdf}} \right) = \left( 1- \frac{1}{t} \right) F\left( \raisebox{-.1 in}{\includegraphics[scale=0.3]{uncross.pdf}} \right).\]

We refer to the invariant $F$ as the \textit{Aicardi-Juyumaya invariant} for colored links. 

Solving for  $\frac{1}{w} F\left( \raisebox{-.1 in}{\reflectbox{\includegraphics[scale=.3]{negcross.pdf}}} \right)$ in relation (IV) and substituting it into relation (III), yields the following skein relation:
\[
(*)  \,\,  \frac{1}{w} F\left( \raisebox{-.1 in}{\reflectbox{\includegraphics[scale=.3]{poscrosscolor_reflect.pdf}}}\right)-w F\left( \raisebox{-.1 in}{\includegraphics[scale=.3]{negcrosscolor.pdf}}\right) = \left(t-\frac{1}{t} \right) F\left( \raisebox{-.1 in}{\includegraphics[scale=.3]{uncross.pdf}} \right)+ tw F\left( \raisebox{-.1 in}{\includegraphics[scale=.3]{negcross.pdf}} \right) -\frac{1}{tw}F \left( \raisebox{-.1 in}{\reflectbox{\includegraphics[scale=.3]{negcross.pdf}}}\right). 
\]

The HOMFLY-PT polynomial $P$ of a classical (single-colored) link is valued in $\mathbb{Z}[\ell, \ell^{-1}, m]$ and is uniquely defined by the condition $P \left( \raisebox{-4pt}{\includegraphics[scale=0.2]{unknot.pdf}} \right) = 1$ and the skein relation below (see~\cite{A2016}):
 \[
 \ell \,P \left( \raisebox{-.1 in}{\reflectbox{\includegraphics[scale=.3]{negcross.pdf}}}\right) + \ell^{-1}\, P \left( \raisebox{-.1 in}{\includegraphics[scale=.3]{negcross.pdf}} \right)
+   m \, P\left( \raisebox{-.1 in}{\includegraphics[scale=.3]{uncross.pdf}} \right) =0.\]
If $L$ is a knot or a link whose components are colored with the same color, then $F(L)$ is determined by relations (I) and (IV), and it coincides with the polynomial $P(L)$ after the following substitutions:
\[\ell = \frac{i}{w\sqrt{t}} \quad \text{and} \quad m = i\left(\frac{1}{\sqrt{t}}  - \sqrt{t}\right). \]

The scope of this paper is two-fold. We first extend the invariant $F$ to oriented colored singular links and denote the resulting invariant of colored singular links by $[\, \cdot \,]$. We also prove a set of graphical skein relations involving oriented, colored, planar 4-valent graphs;  this is done in Section~\ref{ExtensionSingLinks}. This set of graphical relations provide a recursive way to evaluate oriented, colored, planar graphs with 4-valent vertices. In Section~\ref{state-sum model}, we explain how our graphical skein relations provide a state-sum model for the Aicardi-Juyumaya invariant $F$ for colored links, where the states associated to a diagram of a colored link are oriented, colored, $4$-valent planar graphs. This also allows us to define the invariant $[\, \cdot \,]$ for colored singular link without relying on the existence of the invariant $F$ for colored links.


\section{An extension of the Aicardi-Juyumaya invariant to colored singular links} \label{ExtensionSingLinks}

In this section we extend the Aicardi-Juyumaya invariant to oriented colored singular links. We denote the resulting invariant by $[ \, \cdot \,]$. We impose first that 
$[ D] = F(D)$, whenever $D$ is a colored (classical) link diagram. Next, we impose the following skein relations for $[ \, \cdot \,]$:
\begin{align} \label{sing1}
    \left[ \raisebox{-.11 in}{\includegraphics[scale=.3]{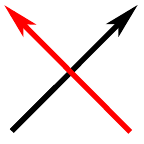}}\right]=& \frac{1}{w}\left[ \raisebox{-.11 in}{\reflectbox{\includegraphics[scale=.3]{poscrosscolor_reflect.pdf}}}\right] + \frac{1}{t}\left[ \raisebox{-.11 in}{\includegraphics[scale=.3]{uncross.pdf}}\right] + \frac{1}{tw} \left[ \raisebox{-.11 in}{\reflectbox{\includegraphics[scale=.3]{negcross.pdf}}} \right]  
    \end{align}
    \begin{align}\label{sing2}
    \left[ \raisebox{-.11 in}{\includegraphics[scale=.3]{crosscolor.pdf}} \right]=& w \left[ \raisebox{-.1in}{{\includegraphics[scale=.3]{negcrosscolor.pdf}}}\right] + t \left[ \raisebox{-.1 in}{\includegraphics[scale=.3]{uncross.pdf}}\right] + tw \left[ \raisebox{-.1 in}{\includegraphics[scale=.3]{negcross.pdf}}\right].
\end{align}
That is, if $\tilde{D}$ is a diagram of a colored singular link, we can use either the skein relation~\eqref{sing1} or relation~\eqref{sing2} at each of the singular crossings in $\tilde{D}$, to write $[\tilde{D}]$ as a $\mathbb{Q}(x, w, t)$-linear combination of evaluations of colored links. We note that the each diagram in relations~\eqref{sing1} and~\eqref{sing2} is a part of a larger diagram of a colored singular link whose components have colorations that are globally compatible; in particular, the colorations of the strands shown in a diagram determine the colorations of the link components that the strands are part of. When forgetting the colorations, the diagrams in the two sides of each of these skein relations are parts of larger diagrams that are identical outside of the small neighborhood where the skein relation is applied.

Subtracting relation~\eqref{sing2} from relation~\eqref{sing1}, yields the skein relation
\begin{align} \label{posminusnegcolor}
    \frac{1}{w} \left[ \raisebox{-.1 in}{\reflectbox{\includegraphics[scale=.3]{poscrosscolor_reflect.pdf}}}\right]-w \left[ \raisebox{-.1 in}{\includegraphics[scale=.3]{negcrosscolor.pdf}}\right] = \left(t-\frac{1}{t} \right)\left[ \raisebox{-.1 in}{\includegraphics[scale=.3]{uncross.pdf}} \right]+ tw \left[ \raisebox{-.1 in}{\includegraphics[scale=.3]{negcross.pdf}} \right] -\frac{1}{tw} \left[ \raisebox{-.1 in}{\reflectbox{\includegraphics[scale=.3]{negcross.pdf}}}\right], 
\end{align}
which we know that $[ \, \cdot \,]$ satisfies, since the invariant $F$ satisfies an equivalent skein relation, as shown in Section~\ref{intro}, and $F(D) = [D]$, for any diagram $D$ of a colored link. Hence, the evaluation $[\tilde{D}]$ of a colored singular link diagram is independent on whether we use the skein relation~\eqref{sing1} or~\eqref{sing2}.

The skein relations~\eqref{sing1} or~\eqref{sing2} have the following equivalent forms, which will be used later in this section.
\begin{align} 
    \frac{1}{w}\left[ \raisebox{-.11 in}{\reflectbox{\includegraphics[scale=.3]{poscrosscolor_reflect.pdf}}}\right]=& - \frac{1}{t}\left[ \raisebox{-.11 in}{\includegraphics[scale=.3]{uncross.pdf}}\right] - \frac{1}{tw} \left[ \raisebox{-.11 in}{\reflectbox{\includegraphics[scale=.3]{negcross.pdf}}} \right] + \left[ \raisebox{-.11 in}{\includegraphics[scale=.3]{crosscolor.pdf}}\right] \label{poscolor} \\
    w \left[ {\raisebox{-.1 in}{\includegraphics[scale=.3]{negcrosscolor.pdf}}}\right] =& -t \left[ \raisebox{-.1 in}{\includegraphics[scale=.3]{uncross.pdf}}\right] - tw \left[ \raisebox{-.1 in}{\includegraphics[scale=.3]{negcross.pdf}}\right] + \left[ \raisebox{-.11 in}{\includegraphics[scale=.3]{crosscolor.pdf}} \right]\label{negcolor}
\end{align}

\begin{theorem}
If $\tilde{D}_1$ and $\tilde{D}_2$ are diagrams representing ambient isotopic colored singular links, then $[\tilde{D}_1] = [\tilde{D}_2]$. That is, $[ \, \cdot \,]$ is an invariant for colored singular links.
\end{theorem}

\begin{proof}
We first note that since $[\, \cdot \,]$ is an extension of $F,$ then $[\, \cdot \,]$ is invariant under Reidemeister moves $R1, R2$ and $R3$. It remains to show that $[\, \cdot \,]$ is also invariant under the moves $R4$ and $R5$. Starting with the Reidemeister-type move $R4$, we find that:
\begin{align*}
    \left[\raisebox{-.1 in}{\includegraphics[scale=.3]{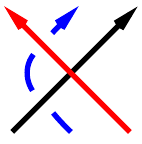}}\right]&\stackrel{\eqref{sing1}}{=}\frac{1}{w}\left[\raisebox{-.1 in}{\includegraphics[scale=.3]{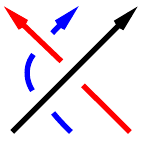}}\right]+\frac{1}{t}\left[\raisebox{-.1 in}{\includegraphics[scale=.3]{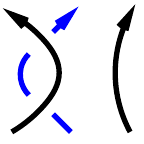}}\right]+\frac{1}{tw}\left[\raisebox{-.1 in}{\includegraphics[scale=.3]{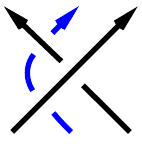}}\right]\\
    &\stackrel{R3}{=}\frac{1}{w}\left[\raisebox{-.1 in}{\includegraphics[scale=.3]{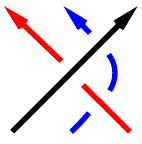}}\right]+\frac{1}{t}\left[\raisebox{-.1 in}{\includegraphics[scale=.3]{R43.pdf}}\right]+\frac{1}{tw}\left[\raisebox{-.1 in}{\includegraphics[scale=.3]{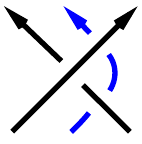}}\right]\\
    &\stackrel{R2}{=}\frac{1}{w}\left[\raisebox{-.1 in}{\includegraphics[scale=.3]{R45.pdf}}\right]+\frac{1}{t}\left[\raisebox{-.1 in}{\includegraphics[scale=.3]{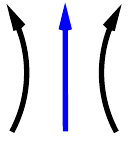}}\right]+\frac{1}{tw}\left[\raisebox{-.1 in}{\includegraphics[scale=.3]{R46.pdf}}\right]\\
    &\stackrel{R2}{=}\frac{1}{w}\left[\raisebox{-.1 in}{\includegraphics[scale=.3]{R45.pdf}}\right]+\frac{1}{t}\left[\raisebox{-.1 in}{\includegraphics[scale=.3]{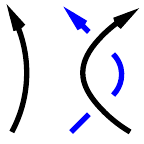}}\right]+\frac{1}{tw}\left[\raisebox{-.1 in}{\includegraphics[scale=.3]{R46.pdf}}\right]
 \stackrel{\eqref{sing1}}{=}\left[\raisebox{-.1 in}{\includegraphics[scale=.3]{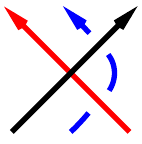}}\right].
\end{align*}
Hence, $\left[\raisebox{-.1 in}{\includegraphics[scale=.3]{R41.pdf}}\right] =\left[\raisebox{-.1 in}{\includegraphics[scale=.3]{R49.pdf}}\right]$. A similar process is used to prove that $\left[\raisebox{-.1 in}{\includegraphics[scale=.3]{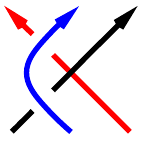}}\right]=\left[\raisebox{-.1 in}{\includegraphics[scale=.3]{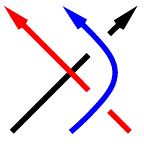}}\right]$ and that $[\, \cdot \,]$ is invariant under the other oriented versions of the move $R4$. We prove next the invariance of $[\, \cdot \,]$ under the Reidemeister-type move $R5$:
\begin{align*}
    \left[\raisebox{-.1 in}{\includegraphics[scale=.3]{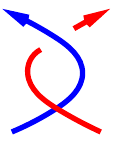}}\right]\stackrel{\eqref{sing2}}{=}w\left[\raisebox{-.1 in}{\includegraphics[scale=.3]{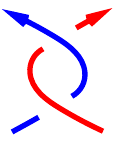}}\right]+t\left[\raisebox{-.1 in}{\includegraphics[scale=.3]{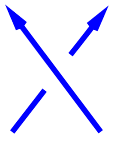}}\right]+tw\left[\raisebox{-.1 in}{\includegraphics[scale=.3]{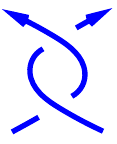}}\right]
    \stackrel{\eqref{sing2}}{=}\left[ \,\raisebox{-.1 in}{\includegraphics[scale=0.3]{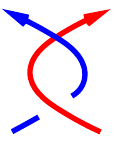}} \, \right ].
\end{align*}
In the first equality above, we applied the skein relation~\eqref{sing2} in a small neighborhood of the 4-valent vertex of the diagram on the left-hand side of the move $R5$. In the second step we applied the same skein relation, this time in reverse order, in a small neighborhood near the top part of the three diagrams obtained in the previous step. The proof of invariance under the other oriented versions of the move $R5$ are done in a similar manner.
This completes the proof that $[\, \cdot \,]$ is an invariant for colored singular links.
\end{proof}

\begin{proposition}
Let $L \, \tilde \cup \, \reflectbox{\raisebox{-.04in}{\includegraphics[scale=.2]{unknot.pdf}}}$ be the disjoint union between a colored link L with the standard diagram of the unknot, where the unknot is colored the same as (at least one) of the components of $L$. Then
\begin{eqnarray}
\left[L \, \tilde \cup \, \reflectbox{\raisebox{-.04in}{\includegraphics[scale=.2]{unknot.pdf}}}\right]  = \frac{tw^2-1}{w(1-t)}\left[L\right] 
    = \left (\frac{tw}{1-t} + \frac{t^{-1}w^{-1}}{1-t^{-1}} \right )\left[L\right]. \label{eq:tied_union}
\end{eqnarray}
\end{proposition}
\begin{proof}
 We represent  $L \, \tilde \cup \, \reflectbox{\raisebox{-.04in}{\includegraphics[scale=.2]{unknot.pdf}}}$ as $\raisebox{-.1 in}{\includegraphics[scale=.3]{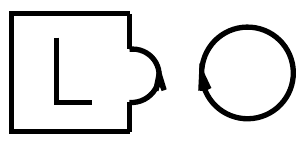}}.$
From relation (IV), we know that 
\[\left(1-t^{-1}\right) \left[\raisebox{-.08 in}{\includegraphics[scale=.25]{uncross.pdf}}\right] = \frac{1}{tw} \left[\reflectbox{\raisebox{-.08 in}{\includegraphics[scale=.25]{negcross.pdf}}}\right] - w \left[\raisebox{-.08 in}{\includegraphics[scale=.25]{negcross.pdf}}\right].\]
Hence,
\begin{align*}
    \left(1-t^{-1}\right) \left[\raisebox{-.1 in}{\includegraphics[scale=.3]{disjunion1.pdf}}\right] &= \frac{1}{tw} \left[\raisebox{-.1 in}{\includegraphics[scale=.3]{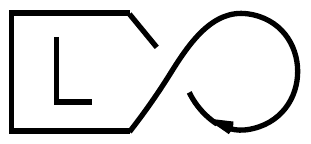}}\right] - w \left[\raisebox{-.1 in}{\includegraphics[scale=.3]{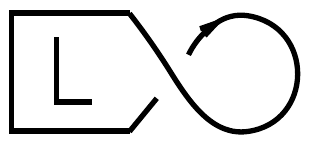}}\right] \\
    \left(1-t^{-1}\right) \left[L \, \tilde \cup \, \reflectbox{\raisebox{-.04in}{\includegraphics[scale=.2]{unknot.pdf}}}\right] &= \left(\frac{1}{tw}-w\right) \left[L\right] \\
    \left[L \, \tilde \cup\, \reflectbox{\raisebox{-.04in}{\includegraphics[scale=.2]{unknot.pdf}}}\right]  &= \frac{tw^2-1}{w(1-t)}\left[L\right] = \left (\frac{tw}{1-t} + \frac{t^{-1}w^{-1}}{1-t^{-1}} \right )\left[L\right].
\end{align*}
Therefore, identity~\eqref{eq:tied_union} holds.
\end{proof}

We make use of the skein relations~\eqref{poscolor} and~\eqref{negcolor} to prove the following graphical skein relations involving oriented, colored, 4-valent planar graphs. An oriented, colored, 4-valent planar graph can be regarded as a diagram of an oriented colored singular link without classical crossings.

\begin{theorem} \label{GraphCalculus}
The following graphical skein relations hold for $[\,\cdot \,]$:
\begin{align}
    \left[ \raisebox{-.1 in}{\includegraphics[scale=.3]{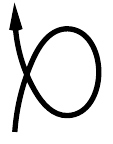}} \right]=& \left( \frac{w}{1-t} + \frac{w^{-1}}{1-t^{-1}}\right) \left[ \raisebox{-.1 in}{\includegraphics[scale=.3]{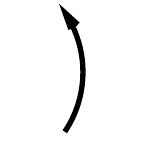}} \right] \label{singloop} \\
    \left[ \reflectbox{\raisebox{-.1 in}{\includegraphics[scale=.3]{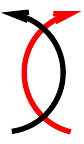}}} \right] =& \left[ \reflectbox{\raisebox{-.1 in}{\includegraphics[scale=.3]{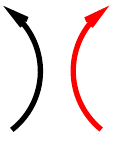}}} \right]+(t+t^{-1})\left[ \raisebox{-.1 in}{\includegraphics[scale=.3]{uncross.pdf}}\right]+(t+t^{-1}) \left[ \raisebox{-.1 in}{\includegraphics[scale=.3]{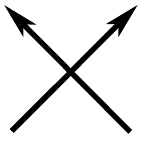}}\right] \label{bigon} \\
    \left[ \reflectbox{\raisebox{-.1 in}{\includegraphics[scale=.3]{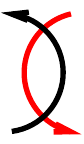}}} \right]=& \left[ \reflectbox{\raisebox{-.1 in}{\includegraphics[scale=.3]{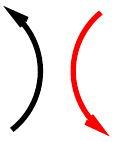}}} \right]+(t+t^{-1}+1) \left[ \reflectbox{\raisebox{-.1 in}{\includegraphics[scale=.3]{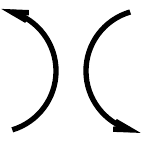}}}\right]+ \left( \frac{wt^{-1}}{1-t}+\frac{w^{-1}t}{1-t^{-1}}\right) \left[ \reflectbox{\raisebox{-.1 in}{\includegraphics[scale=.3]{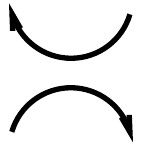}}} \right] \label{bigonop}\\
    \left[ \raisebox{-.1 in}{\includegraphics[scale=.3]{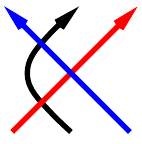}}\right] &- \left[ \reflectbox{\raisebox{-.1 in}{\includegraphics[scale=.3]{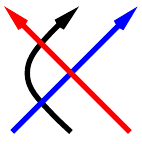}}}\right] = \left[ \raisebox{-.1 in}{\includegraphics[scale=.3]{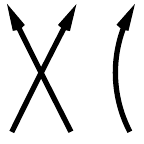}}\right] - \left[ \reflectbox{\raisebox{-.1 in}{\includegraphics[scale=.3]{part1R35.pdf}}}\right]  \label{R3cross}\\
    \left[ \raisebox{-.1 in}{\includegraphics[scale=.3]{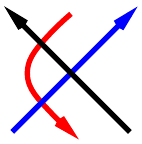}}\right] &- \left[ \reflectbox{\raisebox{-.1 in}{\includegraphics[scale=.3]{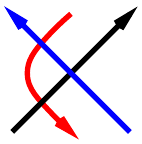}}}\right] = \left(\frac{wt^{-2}}{1-t}+\frac{w^{-1}t^2}{1-t^{-1}}\right)\left(\left[\raisebox{-.1in}{\includegraphics[scale=.3]{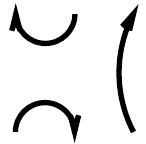}}\right] - \left[\reflectbox{\raisebox{-.1in}{\includegraphics[scale=.3]{part2R3cross9.pdf}}}\right]\right). \label{R3crossdown}
\end{align}
\end{theorem}

\begin{proof}
We start by proving graphical relation \eqref{singloop}. Note that we start by using the skein relation~\eqref{sing1}, with both strands colored similarly. In the second step we use that $[\, \cdot \,]$ is invariant under the Reidemeister move $R1$.
\begin{align*}
    \left[ \raisebox{-.1 in}{\includegraphics[scale=.3]{loop.pdf}} \right] \stackrel{\eqref{sing1}}{=}& \left( \frac{1}{w} + \frac{1}{tw}\right) \left[ \raisebox{-.1 in}{\includegraphics[scale=.3]{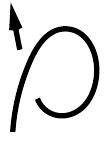}} \right] + \frac{1}{t} \left[ \raisebox{-.1 in}{\includegraphics[scale=.3]{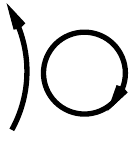}} \right] 
    \stackrel{R1}{=} \left( \frac{1}{w} + \frac{1}{tw}\right) \left[ \raisebox{-.1 in}{\includegraphics[scale=.3]{lineup.pdf}} \right] + \frac{1}{t} \left[ \raisebox{-.1 in}{\includegraphics[scale=.3]{linecirc.pdf}} \right]  \\
      \stackrel{\eqref{eq:tied_union}}{=}& \left( \frac{1}{w} + \frac{1}{tw}\right) \left[ \raisebox{-.1 in}{\includegraphics[scale=.3]{lineup.pdf}} \right] + \frac{1}{t} \left( \frac{tw^2-1}{(1-t)w}\left[ \raisebox{-.1 in}{\includegraphics[scale=.3]{lineup.pdf}} \right] \right) \\
     =& \left( \frac{w}{1-t} + \frac{w^{-1}}{1-t^{-1}}\right) \left[ \raisebox{-.1 in}{\includegraphics[scale=.3]{lineup.pdf}} \right].
\end{align*}

To prove the graphical relation~\eqref{bigon} we start from the diagram on the left side of the relation and apply the skein relation~\eqref{sing1} in a small neighborhood of the top vertex. In the second step, we apply the skein relation~\eqref{sing2} to the resulting first and third diagrams, as we show below.
\begin{align*}
    \left[ \reflectbox{\raisebox{-.1 in}{\includegraphics[scale=.3]{bigoncrosscolor.pdf}}} \right] \stackrel{\eqref{sing1}}{=}& \frac{1}{w} \left[ \reflectbox{\raisebox{-.1 in}{\includegraphics[scale=.3]{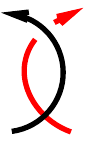}}} \right] + \frac{1}{t} \left[ \raisebox{-.1 in}{\includegraphics[scale=.3]{cross.pdf}} \right] + \frac{1}{tw} \left[ \reflectbox{\raisebox{-.1 in}{\includegraphics[scale=.3]{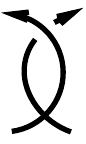}}} \right] \\
    \stackrel{\eqref{sing2}}{=}& \frac{1}{w} \left( w \left[ \reflectbox{\raisebox{-.1 in}{\includegraphics[scale=.3]{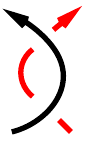}}} \right] + t \left[ \reflectbox{\raisebox{-.1 in}{\includegraphics[scale=.3]{negcross.pdf}}} \right] +tw \left[ \reflectbox{\raisebox{-.1 in}{\includegraphics[scale=.3]{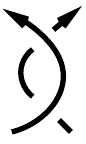}}} \right] \right) + \frac{1}{t} \left[ \raisebox{-.1 in}{\includegraphics[scale=.3]{cross.pdf}} \right] \\
    &\hspace{.2in} + \frac{1}{tw} \left( (w + tw) \left[ \reflectbox{\raisebox{-.1 in}{\includegraphics[scale=.3]{bigonup.pdf}}} \right] + t \left[ \reflectbox{\raisebox{-.1 in}{\includegraphics[scale=.3]{negcross.pdf}}} \right]\right). 
    \end{align*}
 In the next step, we use that $[\, \cdot \, ]$ is invariant under the Reidemeister move $R2$ and combine like terms. Then we apply the skein relation~\eqref{poscolor} to the second diagram with a positive crossing and both strands colored with the same color, and obtain the following calculations.
    \begin{align*}
   \left[ \reflectbox{\raisebox{-.1 in}{\includegraphics[scale=.3]{bigoncrosscolor.pdf}}} \right] \stackrel{R2}{=}& \left[ \reflectbox{\raisebox{-.1 in}{\includegraphics[scale=.3]{uncrosscolor.pdf}}} \right] + \frac{t+1}{w} \left[ \reflectbox{\raisebox{-.1 in}{\includegraphics[scale=.3]{negcross.pdf}}} \right] + \left(t + 1 + t^{-1} \right) \left[ \raisebox{-.1 in}{\includegraphics[scale=.3]{uncross.pdf}} \right] + \frac{1}{t} \left[ \raisebox{-.1 in}{\includegraphics[scale=.3]{cross.pdf}} \right] \\
    \stackrel{\eqref{poscolor}}{=}& \left[ \reflectbox{\raisebox{-.1 in}{\includegraphics[scale=.3]{uncrosscolor.pdf}}} \right]  - \left[ \raisebox{-.1 in}{\includegraphics[scale=.3]{uncross.pdf}} \right] + t \left[ \raisebox{-.1 in}{\includegraphics[scale=.3]{cross.pdf}} \right] + \left(t + 1 + t^{-1} \right) \left[ \raisebox{-.1 in}{\includegraphics[scale=.3]{uncross.pdf}} \right] + \frac{1}{t} \left[ \raisebox{-.1 in}{\includegraphics[scale=.3]{cross.pdf}} \right] \\
    =& \left[ \reflectbox{\raisebox{-.1 in}{\includegraphics[scale=.3]{uncrosscolor.pdf}}} \right] + \left( t + t^{-1} \right) \left[ \raisebox{-.1 in}{\includegraphics[scale=.3]{uncross.pdf}} \right] + \left( t + t^{-1} \right) \left[ \raisebox{-.1 in}{\includegraphics[scale=.3]{cross.pdf}} \right].
\end{align*}
Hence, the graphical relation~\eqref{bigon} holds. In the proof of the graphical relation~\eqref{bigonop}, we use a somewhat similar starting point. Note, however, that the strands in the left side diagram of the relation have different orientations than in the relation we just proved.
\begin{align*}
    \left[ \reflectbox{\raisebox{-.1 in}{\includegraphics[scale=.3]{4valbigoncolordown.pdf}}} \right] \stackrel{\eqref{sing1}}{=}& \frac{1}{w} \left[ \raisebox{-.1 in}{\includegraphics[scale=.3]{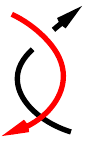}} \right] + \frac{1}{t} \left[ \reflectbox{\rotatebox{180}{\raisebox{-.18 in}{\includegraphics[scale=.3]{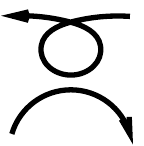}}}} \right] + \frac{1}{tw} \left[ \reflectbox{\rotatebox{180}{{\raisebox{-.18 in}{\includegraphics[scale=.3]{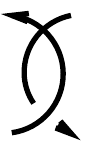}}}}} \right] \\
    \stackrel{\eqref{sing2}, \eqref{singloop}}{=}& \frac{1}{w} \left( w \left[ \raisebox{-.11 in}{\includegraphics[scale=.3]{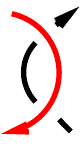}} \right] + t \left[ \reflectbox{\rotatebox{180}{\raisebox{-.18 in}{\includegraphics[scale=.3]{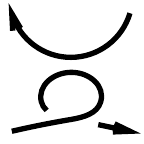}}}} \right] + tw \left[ \reflectbox{\rotatebox{180}{\raisebox{-.18 in}{\includegraphics[scale=.3]{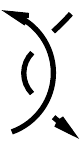}}}} \right] \right) + \frac{1}{t} \left( \frac{w^2-t}{w(1-t)} \left[ \reflectbox{\rotatebox{180}{\raisebox{-.18 in}{\includegraphics[scale=.3]{bigonuncrossdown.pdf}}}} \right] \right) \\
    &\hspace{.2in} + \frac{1}{tw} \left( (wt+w) \left[ \reflectbox{\raisebox{-.1 in}{\includegraphics[scale=.3]{bigondown.pdf}}} \right] + t \left[ \reflectbox{\rotatebox{180}{\raisebox{-.18 in}{\includegraphics[scale=.3]{bigondown1.pdf}}}} \right] \right).
    \end{align*}
   In the second step we also used that $\frac{w^2-t}{w(1-t)} = \frac{w}{1-t} + \frac{w^{-1}}{1-t^{-1}}$. Next we use that $[\, \cdot \, ]$ is invariant under the moves R1 and R2, to obtain the following:
    \begin{align*}
        \left[ \reflectbox{\raisebox{-.1 in}{\includegraphics[scale=.3]{4valbigoncolordown.pdf}}} \right]   \stackrel{R1, R2}{=}& \frac{1}{w} \left( w \left[ \reflectbox{\raisebox{-.1 in}{\includegraphics[scale=.3]{uncrosscolordown.pdf}}} \right] + t \left[ \reflectbox{\raisebox{-.1 in}{\includegraphics[scale=.3]{bigonuncrossdown.pdf}}} \right] + tw \left[ \reflectbox{\raisebox{-.1 in}{\includegraphics[scale=.3]{uncrossdown.pdf}}} \right] \right) + \frac{1}{t} \left( \frac{w^2-t}{w(1-t)} \left[ \reflectbox{\raisebox{-.1 in}{\includegraphics[scale=.3]{bigonuncrossdown.pdf}}} \right] \right) \\
    &\hspace{.2in} + \frac{1}{tw} \left( (wt+w) \left[ \reflectbox{\raisebox{-.1 in}{\includegraphics[scale=.3]{uncrossdown.pdf}}} \right] + t \left[ \reflectbox{\raisebox{-.1 in}{\includegraphics[scale=.3]{bigonuncrossdown.pdf}}} \right] \right) \\
    =& \left[ \reflectbox{\raisebox{-.1 in}{\includegraphics[scale=.3]{uncrosscolordown.pdf}}} \right]+(t+t^{-1}+1) \left[ \reflectbox{\raisebox{-.1 in}{\includegraphics[scale=.3]{uncrossdown.pdf}}}\right]+ \left( \frac{wt^{-1}}{1-t}+\frac{w^{-1}t}{1-t^{-1}}\right) \left[ \reflectbox{\raisebox{-.1 in}{\includegraphics[scale=.3]{bigonuncrossdown.pdf}}} \right].
\end{align*}

We will now evaluate the two terms of the left side of identity~\eqref{R3cross}. Starting with the first diagram, we resolve first the top left 4-valent vertex using the skein relation~\eqref{sing1}. After that, in the first and third of the three resulting diagrams, we resolve the bottom left 4-valent vertex by applying the skein relation~\eqref{sing2}. In the third step we merely simplify expressions. 

\begin{align*}
    \left[ \raisebox{-.1 in}{\includegraphics[scale=.3]{part1R3crossleft.pdf}}\right] \stackrel{\eqref{sing1}}{=}& \frac{1}{w}\left[\raisebox{-.1 in}{\includegraphics[scale=.3]{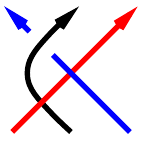}} \right]+ \frac{1}{t} \left[\raisebox{-.1 in}{\includegraphics[scale=.3]{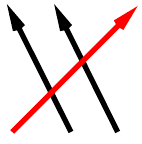}} \right]+ \frac{1}{tw} \left[\raisebox{-.1 in}{\includegraphics[scale=.3]{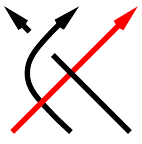}} \right]\\
    \stackrel{\eqref{sing2}}{=}& \frac{1}{w}\left( w\left[\raisebox{-.1 in}{\includegraphics[scale=.3]{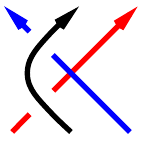}} \right]+t\left[\reflectbox{\raisebox{-.1 in}{\includegraphics[scale=.3]{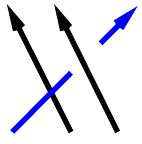}}} \right]+tw\left[\raisebox{-.1 in}{\includegraphics[scale=.3]{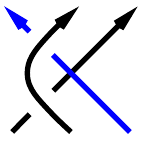}} \right]\right)+\frac{1}{t}\left[\raisebox{-.1 in}{\includegraphics[scale=.3]{part1R3cross2.pdf}} \right]\\
    &\hspace{.2in} + \frac{1}{tw} \left( w\left[\raisebox{-.1 in}{\includegraphics[scale=.3]{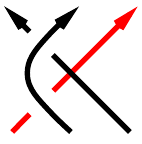}} \right]+t \left[\reflectbox{\raisebox{-.1 in}{\includegraphics[scale=.3]{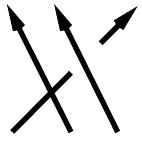}}} \right] +tw\left[\raisebox{-.1 in}{\includegraphics[scale=.3]{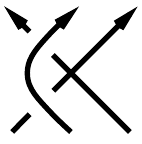}} \right]\right)\\
    =&\left[\raisebox{-.1 in}{\includegraphics[scale=.3]{part1R3cross11.pdf}} \right]+\frac{t}{w}\left[\reflectbox{\raisebox{-.1 in}{\includegraphics[scale=.3]{part1R3cross5ref.pdf}}} \right]+t \left[\raisebox{-.1 in}{\includegraphics[scale=.3]{part1R3cross10.pdf}} \right] +\frac{1}{t}\left[\raisebox{-.1 in}{\includegraphics[scale=.3]{part1R3cross2.pdf}} \right]+\frac{1}{t}\left[\raisebox{-.1 in}{\includegraphics[scale=.3]{part1R3cross8.pdf}} \right] \\
    &\hspace{.2in} +\frac{1}{w}\left[\reflectbox{\raisebox{-.1 in}{\includegraphics[scale=.3]{part1R3cross4.pdf}}} \right]+\left[\raisebox{-.1 in}{\includegraphics[scale=.3]{part1R3cross9.pdf}} \right].
    \end{align*}
 Next we apply skein relation~\eqref{poscolor} at the positive crossing in the second diagram of the resulting sum above.
Finally, in the last step below, we simplify the resulting linear combination of evaluations of colored link diagrams.
 
    \begin{align*}
      \left[ \raisebox{-.1 in}{\includegraphics[scale=.3]{part1R3crossleft.pdf}}\right] \stackrel{\eqref{poscolor}}{=}&
  \left[\raisebox{-.1 in}{\includegraphics[scale=.3]{part1R3cross11.pdf}} \right]+t\left( \frac{-1}{t} \left[\reflectbox{\raisebox{-.1 in}{\includegraphics[scale=.3]{part1R35.pdf}}} \right] -\frac{1}{tw} \left[\reflectbox{\raisebox{-.1 in}{\includegraphics[scale=.3]{part1R3cross4.pdf}}} \right] + \left[\reflectbox{\raisebox{-.1 in}{\includegraphics[scale=.3]{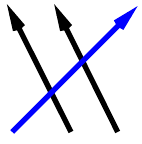}}} \right] \right)+t \left[\raisebox{-.1 in}{\includegraphics[scale=.3]{part1R3cross10.pdf}} \right] \\
    &\hspace{.2in} +\frac{1}{t}\left[\raisebox{-.1 in}{\includegraphics[scale=.3]{part1R3cross2.pdf}} \right]+\frac{1}{t}\left[\raisebox{-.1 in}{\includegraphics[scale=.3]{part1R3cross8.pdf}} \right]+\frac{1}{w}\left[\reflectbox{\raisebox{-.1 in}{\includegraphics[scale=.3]{part1R3cross4.pdf}}} \right]+\left[\raisebox{-.1 in}{\includegraphics[scale=.3]{part1R3cross9.pdf}} \right] \\
    =& \left[\raisebox{-.1 in}{\includegraphics[scale=.3]{part1R3cross11.pdf}} \right] - \left[\reflectbox{\raisebox{-.1 in}{\includegraphics[scale=.3]{part1R35.pdf}}} \right] + t \left[\reflectbox{\raisebox{-.1 in}{\includegraphics[scale=.3]{part1R3cross1ref.pdf}}} \right] + \frac{1}{t} \left[\raisebox{-.1 in}{\includegraphics[scale=.3]{part1R3cross2.pdf}} \right] + t \left[\raisebox{-.1 in}{\includegraphics[scale=.3]{part1R3cross10.pdf}} \right] \\
    &\hspace{.2in} + \frac{1}{t} \left[\raisebox{-.1 in}{\includegraphics[scale=.3]{part1R3cross8.pdf}} \right] + \left[\raisebox{-.1 in}{\includegraphics[scale=.3]{part1R3cross9.pdf}} \right]. 
      \end{align*}
      Similar steps are applied to the second diagram appearing on the left side of the identity~\eqref{R3cross}, as we show below.
    \begin{align*}
    \left[ \reflectbox{\raisebox{-.1 in}{\includegraphics[scale=.3]{part1R3crossrightref.pdf}}}\right] \stackrel{\eqref{sing2}}{=}& w\left[ \reflectbox{\raisebox{-.1 in}{\includegraphics[scale=.3]{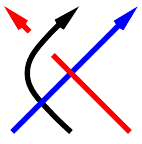}}}\right]+t\left[\reflectbox{\raisebox{-.1 in}{\includegraphics[scale=.3]{part1R3cross1ref.pdf}}} \right]+tw\left[ \reflectbox{\raisebox{-.1 in}{\includegraphics[scale=.3]{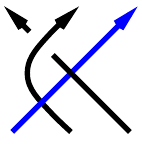}}}\right]\\
    \stackrel{\eqref{sing1}}{=}& w\left( \frac{1}{w}\left[ \reflectbox{\raisebox{-.1 in}{\includegraphics[scale=.3]{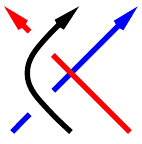}}}\right]+\frac{1}{t}\left[\raisebox{-.1 in}{\includegraphics[scale=.3]{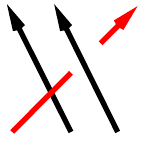}} \right]+\frac{1}{tw}\left[ \reflectbox{\raisebox{-.1 in}{\includegraphics[scale=.3]{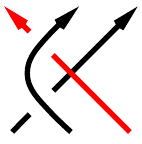}}}\right]\right)+t\left[\reflectbox{\raisebox{-.1 in}{\includegraphics[scale=.3]{part1R3cross1ref.pdf}}} \right]\\
    &\hspace{.2in} +tw\left( \frac{1}{w}\left[ \reflectbox{\raisebox{-.1 in}{\includegraphics[scale=.3]{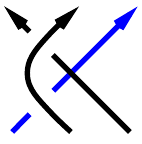}}}\right]+\frac{1}{t} \left[\raisebox{-.1 in}{\includegraphics[scale=.3]{part1R3cross4.pdf}} \right] +\frac{1}{tw}\left[ \reflectbox{\raisebox{-.1 in}{\includegraphics[scale=.3]{part1R3cross9.pdf}}}\right]\right)\\
    =&\left[ \reflectbox{\raisebox{-.1 in}{\includegraphics[scale=.3]{part1R3cross15ref.pdf}}}\right]+\frac{w}{t}\left[\raisebox{-.1 in}{\includegraphics[scale=.3]{part1R3cross3.pdf}} \right]+\frac{1}{t} \left[ \reflectbox{\raisebox{-.1 in}{\includegraphics[scale=.3]{part1R3cross16ref.pdf}}}\right] +t\left[\reflectbox{\raisebox{-.1 in}{\includegraphics[scale=.3]{part1R3cross1ref.pdf}}} \right]+t\left[ \reflectbox{\raisebox{-.1 in}{\includegraphics[scale=.3]{part1R3cross14ref.pdf}}}\right] \\
    &\hspace{.2in} +w\left[\raisebox{-.1 in}{\includegraphics[scale=.3]{part1R3cross4.pdf}} \right]+\left[ \reflectbox{\raisebox{-.1 in}{\includegraphics[scale=.3]{part1R3cross9.pdf}}}\right] \\
    \stackrel{\eqref{negcolor}}{=}&\left[ \reflectbox{\raisebox{-.1 in}{\includegraphics[scale=.3]{part1R3cross15ref.pdf}}}\right]+\frac{1}{t} \left( -t \left[ \raisebox{-.1 in}{\includegraphics[scale=.3]{part1R35.pdf}}\right] -tw \left[\raisebox{-.1 in}{\includegraphics[scale=.3]{part1R3cross4.pdf}} \right] + \left[\raisebox{-.1 in}{\includegraphics[scale=.3]{part1R3cross2.pdf}} \right] \right) +\frac{1}{t} \left[ \reflectbox{\raisebox{-.1 in}{\includegraphics[scale=.3]{part1R3cross16ref.pdf}}}\right] \\
    &\hspace{.2in} +t\left[\reflectbox{\raisebox{-.1 in}{\includegraphics[scale=.3]{part1R3cross1ref.pdf}}} \right] +t\left[ \reflectbox{\raisebox{-.1 in}{\includegraphics[scale=.3]{part1R3cross14ref.pdf}}}\right]+w\left[\raisebox{-.1 in}{\includegraphics[scale=.3]{part1R3cross4.pdf}} \right]+\left[ \reflectbox{\raisebox{-.1 in}{\includegraphics[scale=.3]{part1R3cross9.pdf}}}\right].
        \end{align*}
Therefore, 
        \begin{align*}
  \left[ \reflectbox{\raisebox{-.1 in}{\includegraphics[scale=.3]{part1R3crossrightref.pdf}}}\right]   
    =& \left[ \reflectbox{\raisebox{-.1 in}{\includegraphics[scale=.3]{part1R3cross15ref.pdf}}}\right] - \left[ \raisebox{-.1 in}{\includegraphics[scale=.3]{part1R35.pdf}}\right] + \frac{1}{t} \left[\raisebox{-.1 in}{\includegraphics[scale=.3]{part1R3cross2.pdf}} \right] + t \left[\reflectbox{\raisebox{-.1 in}{\includegraphics[scale=.3]{part1R3cross1ref.pdf}}} \right] + \frac{1}{t} \left[ \reflectbox{\raisebox{-.1 in}{\includegraphics[scale=.3]{part1R3cross16ref.pdf}}}\right] \\
    &\hspace{.2in}+ t \left[ \reflectbox{\raisebox{-.1 in}{\includegraphics[scale=.3]{part1R3cross14ref.pdf}}}\right] + \left[ \reflectbox{\raisebox{-.1 in}{\includegraphics[scale=.3]{part1R3cross9.pdf}}}\right].
    \end{align*}

Combining the two results and using that $[\, \cdot \,]$ is invariant under the move $R4$, we obtain the following equality, and thus the graphical skein relation~\eqref{R3cross} holds:
    \[
    \left[ \raisebox{-.1 in}{\includegraphics[scale=.3]{part1R3crossleft.pdf}}\right] - \left[ \reflectbox{\raisebox{-.1 in}{\includegraphics[scale=.3]{part1R3crossrightref.pdf}}}\right] = \left[ \raisebox{-.1 in}{\includegraphics[scale=.3]{part1R35.pdf}}\right] - \left[ \reflectbox{\raisebox{-.1 in}{\includegraphics[scale=.3]{part1R35.pdf}}}\right].
    \]

It remains to prove the graphical skein relation~\eqref{R3crossdown} and we start by evaluating each of the colored graph diagrams on the left side of the relation. Starting with the first diagram, we obtain the following:
\begin{align*}
    \left[ \raisebox{-.1 in}{\includegraphics[scale=.3]{part2R3crossleft.pdf}}\right] \stackrel{\eqref{sing1}}{=} &\frac{1}{w}\left[\raisebox{-.1in}{\includegraphics[scale=.3]{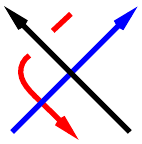}}\right] + \frac{1}{t}\left[\raisebox{-.1in}{\includegraphics[scale=.3]{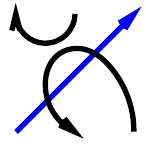}}\right] + \frac{1}{tw}\left[\raisebox{-.1in}{\includegraphics[scale=.3]{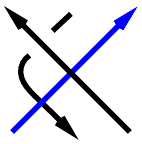}}\right] \\
    \stackrel{\eqref{sing2}, \eqref{bigonop}}{=}& \frac{1}{w}\left(w\left[\raisebox{-.1in}{\includegraphics[scale=.3]{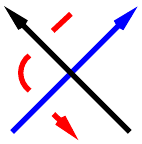}}\right] + t\left[\raisebox{-.1in}{\includegraphics[scale=.3]{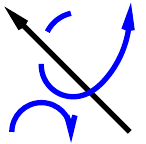}}\right] + tw\left[\raisebox{-.1in}{\includegraphics[scale=.3]{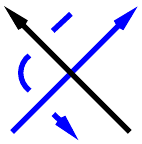}}\right]\right) \\ &\hspace{.2in} + \frac{1}{t}\left(\left[\raisebox{-.1in}{\includegraphics[scale=.3]{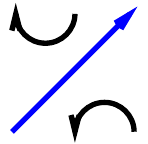}}\right] + \left(t+ t^{-1}+1\right)\left[\raisebox{-.1in}{\includegraphics[scale=.3]{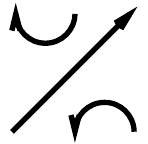}}\right]\right) \\ &\hspace{.2in} + \frac{1}{t} \left(\frac{wt^{-1}}{1-t}+\frac{w^{-1}t}{1-t^{-1}}\right)\left[\raisebox{-.1in}{\includegraphics[scale=.3]{part2R3cross9.pdf}}\right]
     \\ &\hspace{.2in} + \frac{1}{tw}\left(w\left[\raisebox{-.1in}{\includegraphics[scale=.3]{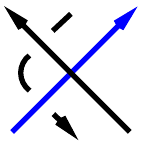}}\right] + t\left[\raisebox{-.1in}{\includegraphics[scale=.3]{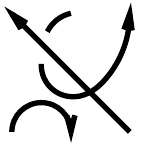}}\right] + tw\left[\raisebox{-.1in}{\includegraphics[scale=.3]{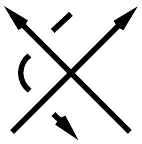}}\right]\right).
       \end{align*}
       By simplifying the result above, we have:
\begin{align*}    
    \left[ \raisebox{-.1 in}{\includegraphics[scale=.3]{part2R3crossleft.pdf}}\right] 
    =& \left[\raisebox{-.1in}{\includegraphics[scale=.3]{part2R3cross4.pdf}}\right] + t\left[\raisebox{-.1in}{\includegraphics[scale=.3]{part2R3cross6.pdf}}\right] + \frac{1}{t}\left[\raisebox{-.1in}{\includegraphics[scale=.3]{part2R3cross10.pdf}}\right] + \left[\raisebox{-.1in}{\includegraphics[scale=.3]{part2R3cross12.pdf}}\right] + \frac{t}{w}\left[\raisebox{-.1in}{\includegraphics[scale=.3]{part2R3cross5.pdf}}\right] \\ &\hspace{.2in} + \frac{1}{w}\left[\raisebox{-.1in}{\includegraphics[scale=.3]{part2R3cross11.pdf}}\right] + \frac{1}{t}\left[\raisebox{-.1in}{\includegraphics[scale=.3]{part2R3cross7.pdf}}\right] + \left(1+ t^{-2}+ t^{-1}\right)\left[\raisebox{-.1in}{\includegraphics[scale=.3]{part2R3cross8.pdf}}\right] \\ &\hspace{.2in} + \left(\frac{wt^{-2}}{1-t}+\frac{w^{-1}}{1-t^{-1}}\right)\left[\raisebox{-.1in}{\includegraphics[scale=.3]{part2R3cross9.pdf}}\right].
    \end{align*}
Continuing on with this same relation, we apply skein relation~\eqref{sing2} to the fifth and sixth diagrams on the right-hand side of the above equality, followed by applications of the invariance of $[ \,\cdot \,]$ under the Reidemeister moves $R1$ and $R2$, to obtain the following:
    \begin{align*}
    \left[ \raisebox{-.1 in}{\includegraphics[scale=.3]{part2R3crossleft.pdf}}\right] \stackrel{\eqref{sing2}, R1, R2}{=}& \left[\raisebox{-.1in}{\includegraphics[scale=.3]{part2R3cross4.pdf}}\right] + t\left[\raisebox{-.1in}{\includegraphics[scale=.3]{part2R3cross6.pdf}}\right] + \frac{1}{t}\left[\raisebox{-.1in}{\includegraphics[scale=.3]{part2R3cross10.pdf}}\right] + \left[\raisebox{-.1in}{\includegraphics[scale=.3]{part2R3cross12.pdf}}\right] \\ &\hspace{.2in} + \left(t\left[\raisebox{-.1in}{\includegraphics[scale=.3]{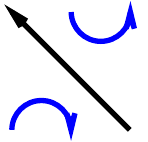}}\right] + \frac{t^2}{w}\left[\raisebox{-.1in}{\includegraphics[scale=.3]{part2R3cross9.pdf}}\right] + t^2\left[\reflectbox{\raisebox{-.1in}{\includegraphics[scale=.3]{part2R3cross8.pdf}}}\right]\right) \\ &\hspace{.2in} + \left(\left(t+1\right)\left[\reflectbox{\raisebox{-.1in}{\includegraphics[scale=.3]{part2R3cross8.pdf}}}\right] + \frac{t}{w}\left[\raisebox{-.1in}{\includegraphics[scale=.3]{part2R3cross9.pdf}}\right]\right) \\ &\hspace{.2in} + \frac{1}{t}\left[\raisebox{-.1in}{\includegraphics[scale=.3]{part2R3cross7.pdf}}\right] + \left(1+ t^{-2}+ t^{-1} \right)\left[\raisebox{-.1in}{\includegraphics[scale=.3]{part2R3cross8.pdf}}\right] \\ &\hspace{.2in} + \left(\frac{wt^{-2}}{1-t}+\frac{w^{-1}}{1-t^{-1}}\right)\left[\raisebox{-.1in}{\includegraphics[scale=.3]{part2R3cross9.pdf}}\right]. \end{align*}
Finally, by simplifying the above equality, we have:
\begin{align*}    
   \left[ \raisebox{-.1 in}{\includegraphics[scale=.3]{part2R3crossleft.pdf}}\right]  
   =& \left[\raisebox{-.1in}{\includegraphics[scale=.3]{part2R3cross4.pdf}}\right] + t\left[\raisebox{-.1in}{\includegraphics[scale=.3]{part2R3cross6.pdf}}\right] + \frac{1}{t}\left[\raisebox{-.1in}{\includegraphics[scale=.3]{part2R3cross10.pdf}}\right] + \left[\raisebox{-.1in}{\includegraphics[scale=.3]{part2R3cross12.pdf}}\right] + (t^2+ t+1) \left[\reflectbox{\raisebox{-.1in}{\includegraphics[scale=.3]{part2R3cross8.pdf}}}\right] \\ &\hspace{.2in} + \left(1+ t^{-2} + t^{-1} \right)\left[\raisebox{-.1in}{\includegraphics[scale=.3]{part2R3cross8.pdf}}\right] + \frac{1}{t}\left[\raisebox{-.1in}{\includegraphics[scale=.3]{part2R3cross7.pdf}}\right] + t\left[\raisebox{-.1in}{\includegraphics[scale=.3]{part2R3cross13.pdf}}\right] \\ &\hspace{.2in} +  \left(\frac{t}{w}+\frac{t^2}{w}+\frac{wt^{-2}}{1-t}+\frac{w^{-1}}{1-t^{-1}}\right)\left[\raisebox{-.1in}{\includegraphics[scale=.3]{part2R3cross9.pdf}}\right].
\end{align*}

By applying similar steps to the second diagram of the left-hand side of the skein relation~\eqref{R3crossdown}, we obtain the following equality:
\begin{align*}
    \left[ \reflectbox{\raisebox{-.1 in}{\includegraphics[scale=.3]{part2R3crossrightref.pdf}}}\right] =&  \left[\raisebox{-.1in}{\includegraphics[scale=.3]{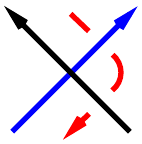}}\right] + t\left[\raisebox{-.1in}{\includegraphics[scale=.3]{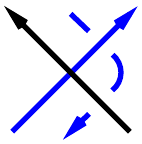}}\right] + \frac{1}{t}\left[\raisebox{-.1in}{\includegraphics[scale=.3]{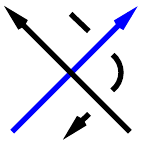}}\right] + \left[\raisebox{-.1in}{\includegraphics[scale=.3]{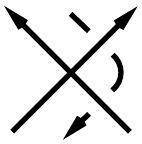}}\right] + \left(t^2+t+1\right)\left[\reflectbox{\raisebox{-.1in}{\includegraphics[scale=.3]{part2R3cross8.pdf}}}\right] \\ &\hspace{.2in} + \left(1+ t^{-2} + t^{-1}\right)\left[\raisebox{-.1in}{\includegraphics[scale=.3]{part2R3cross8.pdf}}\right] + \frac{1}{t}\left[\raisebox{-.1in}{\includegraphics[scale=.3]{part2R3cross7.pdf}}\right] + t\left[\raisebox{-.1in}{\includegraphics[scale=.3]{part2R3cross13.pdf}}\right] \\ &\hspace{.2in} +  \left(\frac{t}{w}+\frac{t^2}{w}+\frac{wt^{-2}}{1-t}+\frac{w^{-1}}{1-t^{-1}}\right)\left[\reflectbox{\raisebox{-.1in}{\includegraphics[scale=.3]{part2R3cross9.pdf}}}\right].
    \end{align*}

 Then, by using that $[\, \cdot \,]$ is invariant under the move $R4$, we obtain the desired skein relation:
   \begin{align*}
    \left[ \raisebox{-.1 in}{\includegraphics[scale=.3]{part2R3crossleft.pdf}}\right] - \left[ \reflectbox{\raisebox{-.1 in}{\includegraphics[scale=.3]{part2R3crossrightref.pdf}}}\right] =&\left( \frac{t}{w}+\frac{t^2}{w}+\frac{wt^{-2}}{1-t}+\frac{w^{-1}}{1-t^{-1}} \right) \left( \left[\raisebox{-.1in}{\includegraphics[scale=.3]{part2R3cross9.pdf}}\right] - \left[\reflectbox{\raisebox{-.1in}{\includegraphics[scale=.3]{part2R3cross9.pdf}}}\right] \right) \\
   =&  \left(\frac{wt^{-2}}{1-t}+\frac{w^{-1}t^2}{1-t^{-1}}\right) \left( \left[\raisebox{-.1in}{\includegraphics[scale=.3]{part2R3cross9.pdf}}\right] - \left[\reflectbox{\raisebox{-.1in}{\includegraphics[scale=.3]{part2R3cross9.pdf}}}\right] \right).
\end{align*}
Hence the graphical skein relation~\eqref{R3crossdown} holds.
\end{proof}

\begin{remark}
We note that if the two strands in the left-hand side diagrams of the skein relations~\eqref{sing1} and~\eqref{sing2} are colored with the same color, then the following holds:
\begin{align*} 
  \left[ \raisebox{-.11 in}{\includegraphics[scale=.3]{cross.pdf}}\right]=  \frac{t+1}{tw}  \left[ \raisebox{-.11 in}{\reflectbox{\includegraphics[scale=.3]{negcross.pdf}}} \right]  + \frac{1}{t}\left[ \raisebox{-.11 in}{\includegraphics[scale=.3]{uncross.pdf}}\right], \,\,\,
   \left[ \raisebox{-.11 in}{\includegraphics[scale=.3]{cross.pdf}}\right] = w(t +1) \left[ \raisebox{-.11 in}{\includegraphics[scale=.3]{negcross.pdf}} \right]  + t \left[ \raisebox{-.11 in}{\includegraphics[scale=.3]{uncross.pdf}}\right]. 
\end{align*}
Solving for  $\left[ \raisebox{-.11in}{\reflectbox{\includegraphics[scale=.3]{negcross.pdf}}} \right]$ and $\left[ \raisebox{-.11in}{\includegraphics[scale=.3]{negcross.pdf}} \right]$ in the above relations and substituting them in~\eqref{sing1} and~\eqref{sing2}, respectively, we obtain the following skein relations:
\begin{align}
    \left[ \reflectbox{\raisebox{-.11 in}{\includegraphics[scale=0.3]{poscrosscolor_reflect.pdf}}} \right] =& \frac{-w}{t+1} \left[ \raisebox{-.11 in}{\includegraphics[scale=0.3]{uncross.pdf}} \right] - \frac{w}{t+1} \left[ \raisebox{-.11 in}{\includegraphics[scale=0.3]{cross.pdf}} \right] + w \left[ \raisebox{-.11 in}{\includegraphics[scale=0.3]{crosscolor.pdf}} \right] \label{graphpos}\\
    \left[ {\raisebox{-.11 in}{\includegraphics[scale=0.3]{negcrosscolor.pdf}}} \right] =& \frac{-t}{w(t+1)} \left[ \raisebox{-.11 in}{\includegraphics[scale=0.3]{uncross.pdf}} \right] - \frac{t}{w(t+1)} \left[ \raisebox{-.11 in}{\includegraphics[scale=0.3]{cross.pdf}} \right] + \frac{1}{w} \left[ \raisebox{-.11 in}{\includegraphics[scale=0.3]{crosscolor.pdf}} \right].\label{graphneg}
\end{align}
\end{remark}

The skein relations~\eqref{graphpos} and~\eqref{graphneg} will play a major role in the following section.


\section{A state-sum model for the Aicardi-Juyumaya invariant for colored links} \label{state-sum model}

The graphical skein relations proved in Theorem~\ref{GraphCalculus} provide a recursive way to evaluate oriented, colored, planar graphs with 4-valent vertices, where a 4-valent vertex is diagrammatically represented as a singular crossing. We denote the evaluation of such planar graph $G$ by $[G]$. Specifically, the graphical skein relations of Theorem~\ref{GraphCalculus} allow the evaluation $[G]$ to be written as a $\mathbb{Q}(x, t, w)$-formal linear combination of evaluations of colored 4-valent planar graphs with fewer vertices. This also allows us to compute the invariant $[\tilde{D}]$ of a colored singular link without relying on the existence of the invariant $F$ for colored classical links. In particular, graphical skein relations of Theorem~\ref{GraphCalculus} provide a state-sum model for the Aicardi-Juyumaya invariant $F$ for colored (classical) links, where the states associated to a diagram of a colored link are oriented, colored, $4$-valent planar graphs, $G$.

Given any colored link diagram $D$, we can resolve each of its crossings using the skein relations~\eqref{graphpos} and~\eqref{graphneg}. After having performed this step at each of the crossings of diagram $D$, we obtain $\left[ D \right]$ written as a formal linear combination of evaluations of its states $G$. Each state $G$ is then evaluated using the graphical relations provided by Theorem~\ref{GraphCalculus}.

\begin{remark}
We remark that this type of approach is not new when working with polynomial invariants in combinatorial knot theory. Kauffman and Vogel~\cite{KV} constructed an invariant for rigid-vertex embeddings of unoriented 4-valent graphs in $\mathbb{R}^3$, which extends the Dubrovnik polynomial to embeddings of graphs. The Dubrovnik polynomial is a version of the 2-variable Kauffman polynomial~\cite{Ka1990} for unoriented links. In~\cite{KV}, a graphical calculus for evaluations of planar 4-valent graphs (not colored) was provided. Then Carpentier~\cite{C} explained how that graphical calculus can be used to provide a state-sum formula for the Dubrovnik polynomial for unoriented links and its extension to rigid-vertex embeddings of 4-valent graphs.
\end{remark}

\begin{lemma} \label{GraphTransform}
If a connected colored 4-valent planar graph contains at least one vertex and does not contain a loop $\raisebox{-.08 in}{\includegraphics[scale=.25]{loop.pdf}} $ or a bigon $\reflectbox{\raisebox{-.08 in}{\includegraphics[scale=.25]{bigoncrosscolor.pdf}}}, \reflectbox{\raisebox{-.08 in}{\includegraphics[scale=.27]{4valbigoncolordown.pdf}}}$, then using a finite sequence of moves of types $\raisebox{-.09 in}{\includegraphics[scale=.27]{part1R3crossleft.pdf}} \leftrightarrow  \reflectbox{\raisebox{-.09 in}{\includegraphics[scale=.25]{part1R3crossrightref.pdf}}}$ and $\raisebox{-.09 in}{\includegraphics[scale=.25]{part2R3crossleft.pdf}} \leftrightarrow \reflectbox{\raisebox{-.09 in}{\includegraphics[scale=.25]{part2R3crossrightref.pdf}}}$, it is possible to transform the graph into a colored 4-valent planar graph that contains a bigon.
\end{lemma}

\begin{proof}
This statement can be proved in a same way as~\cite[Lemma 2]{C}.
\end{proof}

\begin{theorem}
There is a unique rational expression $[ G ] \in \mathbb{Q}(x, w, t)$ associated to a colored 4-valent planar graph $G$, such that $[ \,\cdot \, ]$ satisfies the graphical relations of Theorem~\ref{GraphCalculus}, as well as the following equalities:
\begin{align} \label{RelTrivialLinks}
 \left[ \raisebox{-.07 in}{\includegraphics[scale=0.2]{unknot.pdf}} \right] = 1, \left[  O_n^n   \right] = \left(\frac{1}{wx} \right)^{n-1}, \text{and} \, \left[  O_n^c  \tilde \cup  \raisebox{-.05 in}{\includegraphics[scale=0.2]{unknot.pdf}}   \right] = \frac{tw^2-1}{(1-t)w}  \left[  O_n^c   \right], 
 \end{align}
where $O_n^c$ is the standard diagram of the unlink with $n$ components and $c$ colors ($c \leq n)$, and where the circle in the disjoint union $O_n^c \tilde \cup  \raisebox{-.05 in}{\includegraphics[scale=0.2]{unknot.pdf}}$ has the same color as (at least) one of the unknotted circles in $O_n^c$.
\end{theorem}

\begin{proof}
The proof is by induction on the number of vertices in a colored 4-valent planar graph and is similar to~\cite[Theorem 3]{C}.

Let $[ \,\cdot \,]$ and $[ \,\cdot \,]^*$ be two rational expressions that satisfy the graphical relations given in Theorem~\ref{GraphCalculus} and the additional relations~\eqref{RelTrivialLinks}. 
If $G$ has no vertices, then since $G$ is planar, it is a colored unlink containing $n$ unknotted circles. Then, by relations~\eqref{RelTrivialLinks}, $[G] = [G]^*$.

Suppose that $[ \,\cdot \,]$ and $[ \,\cdot \,]^*$ yield equal values when evaluated on the  same colored 4-valent planar graph with at most $k$ vertices. Let $G$ be a colored 4-valent planar graph with $k+1$ vertices. If $G$ contains a loop $\raisebox{-.08 in}{\includegraphics[scale=.25]{loop.pdf}} $ or a local configuration of type $\reflectbox{\raisebox{-.08 in}{\includegraphics[scale=.25]{bigoncrosscolor.pdf}}}$ or $\reflectbox{\raisebox{-.08 in}{\includegraphics[scale=.27]{4valbigoncolordown.pdf}}}$, then by applying the skein relation~\eqref{singloop}, \eqref{bigon} or~\eqref{bigonop}, respectively, the evaluations $[G]$ and $[G]^*$ are written in terms of evaluations of colored 4-valent planar graphs with fewer vertices, and therefore, by the induction hypothesis,  $[G] = [G]^*$.

If $G$ does not contain a loop ($\raisebox{-.08 in}{\includegraphics[scale=.25]{loop.pdf}} $) or a bigon ($\reflectbox{\raisebox{-.08 in}{\includegraphics[scale=.25]{bigoncrosscolor.pdf}}}$ or $\reflectbox{\raisebox{-.08 in}{\includegraphics[scale=.27]{4valbigoncolordown.pdf}}}$), then by Lemma~\ref{GraphTransform}, there exists a finite sequence of colored 4-valent planar graphs 
\[G =  G_0 \to G_1 \to G_2 \to \dots \to G_s,\]
 where for each $0 \leq i \leq s-1$, $G_{i+1}$ is obtained from $G_i$ by applying the move $\raisebox{-.09 in}{\includegraphics[scale=.27]{part1R3crossleft.pdf}} \leftrightarrow  \reflectbox{\raisebox{-.09 in}{\includegraphics[scale=.25]{part1R3crossrightref.pdf}}}$ or $\raisebox{-.09 in}{\includegraphics[scale=.25]{part2R3crossleft.pdf}} \leftrightarrow \reflectbox{\raisebox{-.09 in}{\includegraphics[scale=.25]{part2R3crossrightref.pdf}}}$, and where $G_s$ contains a bigon of type $\reflectbox{\raisebox{-.08 in}{\includegraphics[scale=.25]{bigoncrosscolor.pdf}}}$ or $\reflectbox{\raisebox{-.08 in}{\includegraphics[scale=.27]{4valbigoncolordown.pdf}}}$. Then, by applying the skein relations~\eqref{R3cross} and~\eqref{R3crossdown} followed by the inductive hypothesis, we have that
\[ [G_i] -[G_{i+1}] =  [G_i]^* -[G_{i+1}]^*, \,\, \text{for all} \,\, 0 \leq i \leq s-1.  \]
Hence, $[G] - [G_s] = [G]^* - [G_s]^*$. But $G_s$ contains a bigon, and by the previous step in the proof, $[G_s] = [G_s]^*$. It follows that $[G] = [G]^*$.

By the principal of mathematical induction, $[ \,\cdot \,]$ and $[ \,\cdot \,]^*$ yield the same quantity when evaluated from a graph $G$ with any number of 4-valent vertices, and thus $[ \,\cdot \,] = [ \,\cdot \,]^*$.
  \end{proof} 

\begin{theorem}
Let $D_1$ and $D_2$ be diagrams representing a colored link. Resolve all crossings in $D_1$ and $D_2$ using the skein relations~\eqref{graphpos} and~\eqref{graphneg}, to write $[D_1]$ and $[D_2]$ as $\mathbb{Q}(x, w, t)$-linear combinations of evaluations of their states. Then evaluate the corresponding states using the graphical skein relations~\eqref{singloop}--\eqref{R3crossdown}, together with identities~ \eqref{RelTrivialLinks}.
Then $[D_1] = [D_2]$; that is, the rational expression $[ \, \cdot \, ] \in \mathbb{Q}(x, w, t)$ is an invariant of colored links. Moreover, $[\, \cdot \, ]$ is equal to Aicardi-Juyumaya invariant $F$.
\end{theorem}

\begin{proof}
The proof that $[D_1] = [D_2]$ when the two diagrams differ by a Reidemeister move $R1$, $R2$ or $R3$ is
by direct computations, which in some sense are the reverse of the proofs of our graphical skein relations in Theorem~\ref{GraphCalculus}.
\end{proof}


\end{document}